\documentclass[12pt,twoside,reqno]{amsart}
\allowdisplaybreaks
\usepackage{bm}

\usepackage{amssymb,amsmath,mathrsfs,amstext,amsthm,amsfonts,amscd,xcolor, tikz-cd}

\usepackage{bbm}
\usepackage{dsfont}

\usepackage{mathrsfs}

\usepackage{amsthm, enumerate}

\usepackage[ansinew]{inputenc}
\usepackage{graphicx}
\usepackage{hyperref}

\newcommand{\R}{\mathbb{R}}

\newcommand{\N}{\mathbb{N}}

\newcommand{\norm}[1]{\lVert#1\rVert} 
\newcommand{\normtwo}[1]{
	{\left\vert\kern-0.25ex\left\vert\kern-0.25ex\left\vert #1
		\right\vert\kern-0.25ex\right\vert\kern-0.25ex\right\vert} }

\newcommand{\Pp}{\mathbb{P}}

\newcommand{\B}{\mathcal{B}}

\newcommand{\E}{\mathscr{E}}

\newcommand{\F}{\mathcal{F}}

\newcommand\restr[2]{{
		\left.\kern-\nulldelimiterspace 
		#1 
		\vphantom{\big|} 
		\right|_{#2} 
}}

\theoremstyle{plain}
\newtheorem{theorem}{Theorem}[section]
\newtheorem{proposition}{Proposition}[section]
\newtheorem{corollary}[proposition]{Corollary}
\newtheorem{lemma}[proposition]{Lemma}

\theoremstyle{definition}
\newtheorem{definition}{Definition}[section]

\theoremstyle{definition}
\newtheorem{remark}{Remark}[section]

\numberwithin{equation}{section}

\newcommand{\Prob}{\mathrm{Prob}}

\newcommand{\EE}{\mathbb{E}}

\newcommand{\Qop}{\mathcal{Q}}

\title[Large Deviations for Unbounded Observables]{Large Deviations for Unbounded Observables in Dynamical Systems}

\date{}

\begin{document}

\author[A. Pontes]{Anselmo Pontes}
\address{Departamento de Matem\'atica, Pontif\'icia Universidade Cat\'olica do Rio de Janeiro (PUC-Rio), Brazil}
\email{aspjunior2011@mat.puc-rio.br}

\begin{abstract}
In this paper we establish a large deviations type estimate for strongly mixing Markov chains with respect to the $L_p$ norm. As applications we derive such estimates for the iterates of a locally constant random cocycle with mixed rank, as well as for unbounded observables of expanding maps. 
\end{abstract}

\maketitle


\section{Introduction}\label{intro}

A central theme in ergodic theory and probability is the study of statistical properties of dynamical systems and Markov processes. Strong mixing often ensures stochastic behavior of time averages of observables, reflected in classical limit theorems such as the law of large numbers, the central limit theorem (CLT), and large deviations principles (LDP). These results not only clarify the ergodic behavior of deterministic and random systems but also play a fundamental role in applications.

For bounded observables belonging to Banach spaces with spectral gap properties (e.g.\ H\"older or bounded variation functions), large deviations estimates are by now well understood. If $\varphi$ is such an observable for a system with exponential mixing, one typically obtains exponential bounds of the form
\[
\mu \left\{ x : \Big| \frac{1}{n}\sum_{k=0}^{n-1}\varphi(T^k x) - \int \varphi \, d\mu \Big| > \epsilon \right\}
    \leq C e^{-c(\epsilon) n}, \qquad n \geq 1.
\]
This setting has been extensively studied; see for instance \cite{Go69,Go78,MN08,CDK-paper3}.

A natural question is whether such results extend to \emph{unbounded} observables, which frequently arise in concrete problems. Important examples include logarithmic observables near singularities, such as $\varphi(x) = -\log d(x,z)$, which appear in the study of recurrence rates, entropy production, Lyapunov exponents, and continued fraction expansions. In these cases, the lack of boundedness prevents direct application of spectral methods, and obtaining sharp deviation estimates requires additional techniques. Existing works in this direction either focus on particular systems or yield suboptimal bounds, typically with stretched-exponential decay rates such as $\exp(-n^{1/5})$; see for example \cite{AFLV},\cite{NT2020}.

\medskip
\noindent
\textbf{Contributions.}
The main goal of this paper is to develop a systematic framework for large deviation estimates for unbounded observables in the setting of stochastic dynamical systems with strong mixing. Our main contributions are:
\begin{itemize}
    \item We prove a general large deviation estimate (Theorem~\ref{thm:unbounded-LDT}) for Markov systems $(M,K,\mu,\mathcal{B})$, extending the exponential deviation bounds known for bounded observables to a broad class of unbounded observables satisfying suitable truncation and tail conditions.
    \item As an application, we establish new deviation bounds for expanding interval maps and logarithmic observables of the form $\varphi(x) = |\log d(x,z)|$ (Theorem~\ref{thm:LDT expanding}). We employ maximal inequalities for stationary sequences to improve previously known results by showing that the deviation probability decays at least as fast as $\exp(-c \, n^{1/3})$.
    \item We provide concrete examples, including the tent map, to illustrate how the abstract results translate into quantitative bounds in explicit dynamical systems.
\end{itemize}

\medskip
\noindent
\textbf{Structure of the paper.}
In Section 2  we recall the necessary definitions and mixing framework for Markov systems. Section 3 contains the proof of Theorem~\ref{thm:unbounded-LDT}. In Section 4 we apply these results to expanding maps of the interval, proving Theorem~\ref{thm:LDT expanding} and its corollaries. 
  

In this section we recall the main definitions we will use.

   \section{Preliminaries}
  
  In this section we recall the basic definitions and framework.
  
  Let $M$ be a compact metric space with Borel $\sigma$-algebra $\F$, and denote by $\Prob(M)$ the space of Borel probability measures on $M$, endowed with the weak* topology.
  
  \begin{definition}
  	A \emph{stochastic dynamical system} (SDS) on $M$ is a continuous Markov kernel 
  	\[
  	K : M \to \Prob(M), \qquad x \mapsto K_x,
  	\]
  	assigning to each $x\in M$ a probability measure $K_x$ on $M$.
  \end{definition}
  
  \begin{definition}
  	A measure $\mu \in \Prob(M)$ is called \emph{$K$--stationary} if
  	\[
  	\mu(A) = \int_M K_x(A) \, d\mu(x), \qquad \forall A \in \F.
  	\]
  \end{definition}
  
  \begin{remark}
  	The set of $K$--stationary measures is a nonempty compact convex subset of $\Prob(M)$.
  \end{remark}
  
  \begin{definition}
  	The \emph{Markov operator} associated with $K$ is the linear operator
  	\[
  	\Qop = \Qop_K : L^\infty(M) \to L^\infty(M), 
  	\qquad 
  	\Qop \varphi(x) = \int_M \varphi(y) \, dK_x(y).
  	\]
  \end{definition}
  
  We call a triple $(M,K,\mu)$ a \emph{Markov system}.  
  On the product space $X^+ = M^{\mathbb{N}}$ let
  \[
  Z_n(x) := x_n, \qquad x = \{x_n\}_{n\in\N}\in X^+,
  \]
  be the canonical coordinate maps. By Kolmogorov's extension theorem, for each $\pi\in \Prob(M)$ there exists a unique measure $\mathbb{P}_\pi$ on $X^+$ such that $\{Z_n\}_{n\geq 0}$ is a Markov chain with transition kernel $K$ and initial law $\pi$.  
  If $\pi$ is $K$--stationary then $\mathbb{P}_\pi$ is invariant under the left shift
  \[
  \sigma : X^+ \to X^+,\qquad \sigma(\{x_n\}_{n\in\N}) = \{x_{n+1}\}_{n\in\N}.
  \]
  
  Given an observable $\varphi : M \to \R$ and $n\in\N$, the \emph{$n$-th stochastic Birkhoff sum} is
  \[
  S_n \varphi(x) = \sum_{k=0}^{n-1} \varphi(Z_k(x)) = \varphi(x_0) + \cdots + \varphi(x_{n-1}), \qquad x\in X^+.
  \]
  
  The main property that we want to investigate is the rate of convergence in probability of the stochastic Birkhoff sums to the integral of the observable. More precisely, we want to establish finitary concentration upper bounds in the same fashion as in the Hoeffding's inequality.
  
  \begin{definition}
  	We say that a Markov system $(M,K,\mu,\B)$ satisfies a \emph{large deviation theorem (LDT)} if for every $\epsilon>0$ and $\varphi \in \B$ there exist $n_0(\epsilon)\in\N$ and $c(\epsilon)>0$ such that
  	\[
  	\mathbb{P}_\mu \left\{ x \in X^+ : 
  	\Big|\tfrac{1}{n} S_n\varphi(x) - \int \varphi \, d\mu \Big| > \epsilon \right\} 
  	\leq e^{-c(\epsilon) n}, \qquad n\geq n_0(\epsilon).
  	\]
  \end{definition}

  To obtain such type of result it is necessary to restrict the space of observables $\varphi$ to a Banach space $\B \subseteq L^\infty(M)$ with suitable mixing properties. In \cite{Glynn2002}, the authors prove an LDT for uniformly ergodic Markov chains which is known to be the best type of mixing. There are many relevant examples which are not uniformly ergodic, for instance, the process generated by the iterates of an uniformly hyperbolic map. For these maps, there are many results (see \cite{MN08},\cite{Alves}) which derive LDT from the spectral gap of the transfer operator in some adequate Banach space. Nevertheless, there are systems for which the Markov operator does not have the spectral gap property such as the random torus tranlastions and linear cocycle over it considered in \cite{CDK-paper1},\cite{CDK-paper3}. This motivated the authors of \cite{CDK-paper3} to introduce the following weaker notion of mixing.

  \begin{definition}
  	A Markov system $(M,K,\mu,\B)$ is said to be \emph{strongly mixing with rate $\{r_n\}$} if for all $\varphi \in \B$,
  	\[
  	\big\| \Qop^n \varphi - \int_M \varphi \, d\mu \big\|_{\infty} \leq r_n \, \|\varphi\|_{\B}, 
  	\qquad n\geq 1.
  	\]
  \end{definition}

  In \cite{CDK-paper3}, the authors prove LDT for strongly mixing Markov systems with rate $r_{n} \searrow 0$.  
  The main goal of this work is to investigate what can be obtained for unbounded observables. In order to do this, we need to apply the results already known in the bounded case to a truncation of the observable. The main novelty of this work is the use of Burkholder's inequality for the remainder (unbounded) part of the observable.

\section{Martingale Approach}\label{martingale}
 In this section we recover the LDT result of \cite{CDK-paper3} under mild extra assumptions, summability of the rate of convergence. We start by recalling the definition of Martingale difference.

\begin{definition}
    Given a filtration of $\sigma-$algebras  $\F_1 \subset \F_2 \subset \dots \subset \F_n \subset \F$ on a probability space $(\Omega, \F, \Pp)$, a sequence of random variables $\{X_i\}_{i=1}^{n}$ is said to be adapted to this filtration if $X_i$ is $\F_i -$ measurable for all $i = 1, \dots , n$.

\end{definition}

\begin{definition}
    Given a filtration of $\sigma-$algebras  $\F_1 \subset \F_2 \subset \dots \subset \F_n \subset \F$ on a probability space $(\Omega, \F, \Pp)$, a sequence of adapted random variables $\{X_i\}_{i=1}^{n}$ is said to be a martingale difference if it satisfies the following conditions:

    \begin{enumerate}
        \item $\EE [X_1] = 0$,
        \item $\EE[X_i| \F_{i-1}] = 0$.
    \end{enumerate}
\end{definition}

\begin{theorem}[Azuma-Hoeffding]\label{Azuma-Hoeffding}
Let $\{X_i\}_{i=1}^{n}$ be a sequence of martingale differences. If there is $C>0$ such that $\|X_i\|_{\infty} \leq C$, then for all $\varepsilon >0$ we have
\[
\Pp \left(\frac{1}{n}\sum_{i=1}^{n} X_i \geq \varepsilon\right) \leq e^{-\frac{n \varepsilon^2}{2 C^2}} \, .
\]
\end{theorem}

Now we precisely define what the martingale sequence is for which we apply the Azuma-Hoeffding inequality. Recall that we are considering a realization $(Z_1, \dots, Z_n)$ of our Markov chain. Let us start by defining the appropriate filtration. For every $i = 1, \dots , n$ let
\begin{equation*}
    \F_{i} = \F(Z_1, \dots, Z_i).
\end{equation*}

\noindent
Then we clearly have $\F_1 \subset \F_2 \subset \dots \subset \F_n$. 

Fix $\varphi \in \B$ such that $\int \varphi \, d\mu = 0$ and define
\begin{equation*}
    \psi := \sum_{i=0}^{\infty} \Qop^{i} \varphi \quad \text{note that} \quad \varphi = \psi - \Qop\psi .
\end{equation*}

\noindent
Finally, for all $i = 1, 2 ,\dots , n $ consider 
\begin{equation*}
    X_i = \psi(Z_{i+1}) - \Qop \psi (Z_i) \, .
\end{equation*}

\noindent
It is easy to see that 
\begin{align*}
    S_n \varphi &= \sum_{i =1}^{n} \varphi(Z_i) =  \sum_{i =1}^{n-1} X_i + \psi(Z_1) - \Qop\psi(Z_{n}) \, .
\end{align*}  

\noindent Now we show that $\{X_i\}_{i=1}^{n}$ is a martingale difference.

\begin{lemma} $\{X_i\}_{i=0}^{n}$ is a finite martingale difference with respect to the filtration $\{\F_i\}_{i=0}^{n}$.
\end{lemma}

\begin{proof}
	By the Markov property, $\E[X_i \mid \F_{i-1}] = \E[\psi(Z_i)\mid Z_{i-1}] - Q\psi(Z_{i-1}) = 0$, hence $\{X_i\}$ is a martingale difference.
\end{proof}

\medskip

The following result corresponds to the case where $p= \infty$ and the decay rate $r_n$ is summable. The proof here is a extension of the transfer operator case considered in \cite{Alves} to general Markov systems.

\begin{theorem}\label{bounded-ldt}
    Let $\varphi \in L^{\infty} (\mu)$ and suppose that 
    \begin{equation*}
        \psi := \sum_{n=0}^{\infty} \Qop^n \varphi \in L^{\infty} (\mu) \, .
    \end{equation*}
    Then for every $\epsilon > 0$ there exists $N(\epsilon,\varphi) = \frac{4}{\epsilon \|\psi\|_{\infty}}$ such that for $n \geq N$
\begin{equation*}
        \mu\left\{x \in M : \left| \frac{1}{n}S_n\varphi-\int_M \varphi d\mu\right| > \varepsilon \right\} \leq 2 e^{-c(\epsilon) \, n},
    \end{equation*}
    where $c(\epsilon) = \frac{\epsilon^2}{8 (||\varphi||_{\infty} + 2 (||\sum \Qop^n \varphi||_{\infty})^2}$.
\end{theorem}

    \begin{proof}
        Without loss of generality suppose that $\int \varphi \, d\mu = 0$.By the definition of $X_i$ it is clear that for $i = 1,2, \dots , n$ we have,

        \begin{equation}
            ||X_i||_{\infty} \leq ||\varphi||_{\infty} + 2 ||\psi||_{\infty}
        \end{equation} 

\noindent
    By the previous lemma we know that  $\{X_i\}_{i=0}^{n}$ is a martingale difference. Then Azuma-Hoeffding inequality is applicable and we get.

    \begin{equation*}
       \mu\left(\frac{1}{n} \left| \sum X_i \right| > \frac{\varepsilon}{2}\right) \leq 2 \exp\left\{ -\frac{\varepsilon^2 n}{8(||\varphi||_{\infty} + 2 ||\psi||_{\infty})^2} \right\}
    \end{equation*}

\noindent
    for all $n \in \N$. In particular for $n \geq N$ where $\frac{2}{N||\psi||_{\infty}}  \leq \frac{\varepsilon}{2}$ we obtain,
    \begin{align*}
    \mu\left(\frac{1}{n} \left| S_n \right| > \varepsilon\right) &\leq \mu\left(\frac{1}{n} \left| \sum X_i \right| + \frac{2}{N}||\psi||_{\infty}> \varepsilon\right) \\
    &\leq \mu\left(\frac{1}{n} \left| \sum X_i \right| > \frac{\varepsilon}{2}\right) \leq 2 \exp\left\{ -\frac{\varepsilon^2 n}{8(||\varphi||_{\infty} + 2 ||\psi||_{\infty})^2} \right\},
    \end{align*}
which establishes the result.
\end{proof}

In this case, we are going to need the following generalization of Burkholder's Inequality for strictly stationary sequences due to Peligrag et al \cite{Maximal_Lp}.

\begin{theorem}
    Let $X_k$ is a strictly stationary sequence such that $\EE[|X_1|^p] < \infty$, $p \geq 2$. Then
    \begin{equation*}
        \norm{S_n}_p \leq C_p \, n^{\frac{1}{2}} \, [\norm{X_1}_p + 240 \sum_{k=1}^{n} k^{-1/2} \, \norm{\EE[X_k | \F_0]}_p].
    \end{equation*}
\end{theorem}

If $X_i$ is a martingale difference, then this result recovers the classical Burkholder inequality.

\begin{theorem}[Burkholder's Inequality]
    Let $\{X_i\}_{i=1}^{n}$ be a sequence of martingale differences adapted to a filtration $\{\mathcal{F}_i\}$. Then, for any $p \geq 2$, there exists a constant $C_p$ such that
    \begin{equation*}
        \norm{S_n}_p \leq C_p \, n^{\frac{1}{2}} \, \norm{X_1}_p.
    \end{equation*}
\end{theorem}

\begin{remark}
    In general $C_p \leq p^p $, for $p \geq 4$ there are better bounds.
\end{remark}

In the Markov system case we consider $\F_i = \sigma(Z_0, \dots, Z_i)$ and note that $\EE[\varphi(Z_k) | \F_0] = \Qop^k \varphi(Z_0)$. Thus we obtain the following corollary which will be the form of Burkhholder's inequality we will actually apply.

\begin{corollary}
    Let $(M,K,\mu,\B)$ be a strongly mixing Markov system with rate $\{r_n\}$, consider a observable $\varphi \in \B$ with $\int \varphi \, d\mu = 0$. Then 
    \begin{center}
        $\norm{S_n\varphi}_p \leq C_p \, n^{\frac{1}{2}} \, [( \norm{\varphi}_p +240 \sum_{k = 1}^{n} \frac{\|\Qop^{k} \varphi\|_2}{\sqrt{k}})]  $.
    \end{center}
\end{corollary}

\subsection{Large Deviations for some Unbounded Observables}

Now we are ready to extend the previous results to handle unbounded observables with good decay of tails distribution.

\begin{theorem}[Large Deviation Estimate for Unbounded Observables]\label{thm:unbounded-LDT}
Let $(M, K, \mu, \mathcal{B})$ be a strongly mixing Markov system, and let $\varphi : M \to \mathbb{R}$ be an observable satisfying the following conditions:
\begin{enumerate}
    \item (Regularity after truncation) For every $M > 0$, the truncated observable $\varphi \cdot \mathbbm{1}_{\{\varphi \leq M\}}$ belongs to $\mathcal{B}$.
    \item (Exponential tails) There exist constants $C_1, \alpha > 0$ such that
    \[
    \mu(\varphi > t) \leq C_1 e^{-\alpha t}, \quad \forall t > 0.
    \]
    \item ($L^2$-control of tails) There exists a constant $C_2 > 0$ such that
    \[
    \sum_{k=1}^n k^{-1/2} \left\| Q^k\left( \varphi \cdot \mathbbm{1}_{\{\varphi > M\}} \right) \right\|_2 \leq C_2, \quad \forall n \in \mathbb{N}.
    \]
\end{enumerate}
Then, for every $\epsilon > 0$, there exist constants $c = c(\alpha) > 0$ and 
\[
n_0(\epsilon) := \left\lceil \frac{16 C \|\varphi^2\|_2}{\epsilon} \right\rceil \in \mathbb{N}
\]
such that, for all $n \geq n_0(\epsilon)$, the Birkhoff sums
\[
S_n \varphi := \sum_{k=0}^{n-1} \varphi(Z_k)
\]
satisfy the large deviation bound
\[
\mathbb{P} \left( \frac{1}{n} S_n \varphi > \epsilon \right) \leq 2 \exp\left( -c \epsilon^{2/3} n^{1/3} \right).
\]
\end{theorem}

\begin{proof}
Without loss of generality, assume $\int \varphi \, d\mu = 0$. For a truncation level $M > 0$, decompose $\varphi$ as
\[
\varphi = \varphi_M + R_M, \quad \text{where} \quad \varphi_M := \varphi \cdot \mathbbm{1}_{\{\varphi \leq M\}}, \quad R_M := \varphi \cdot \mathbbm{1}_{\{\varphi > M\}}.
\]
Then, by the triangle inequality and subadditivity of probability,
\[
\mathbb{P} \left( S_n \varphi > n \epsilon \right)
\leq \mathbb{P} \left( S_n \varphi_M > \frac{n\epsilon}{2} \right) + \mathbb{P} \left( S_n R_M > \frac{n\epsilon}{2} \right).
\]

\medskip
\noindent Since $\varphi_M$ is bounded by $M$ and belongs to $\mathcal{B}$, strong mixing implies sufficient regularity to apply Theorem \ref{bounded-ldt}:
\[
\mathbb{P} \left( S_n \varphi_M > \frac{n\epsilon}{2} \right)
\leq \exp\left( -\frac{n\epsilon^2}{8M^2} \right).
\]

\medskip
\noindent Using Burkholder's inequality for adapted sequences (see \cite{Burkholder}), we control the $L^2$-norm of the tail sum:
\[
\begin{aligned}
\| S_n R_M \|_2^2 
&\leq 4n C \left\| \varphi \cdot \mathbbm{1}_{\{\varphi > M\}} \right\|_2^2 \\
&\leq 4n C \, \|\varphi^2\|_2 \cdot \mu(\varphi > M)^{1/2} \\
&\leq 4n C \, \|\varphi^2\|_2 \cdot e^{-\alpha M / 2},
\end{aligned}
\]
where we used the exponential tail assumption in the last step. Applying Chebyshev's inequality:
\[
\mathbb{P} \left( S_n R_M > \frac{n\epsilon}{2} \right) 
\leq \frac{4\|S_n R_M\|_2^2}{n^2 \epsilon^2} 
\leq \frac{16 C \|\varphi^2\|_2}{n \epsilon^2} e^{-\alpha M / 2}.
\]

\medskip
\noindent Set
\[
M := \frac{\epsilon^{2/3} n^{1/3}}{\sqrt[3]{4}},
\]
so that both error terms decay at the same subexponential rate. Plugging this into the previous bounds gives:
\[
\mathbb{P} \left( S_n \varphi > n\epsilon \right) 
\leq \exp\left( -c_1 \epsilon^{2/3} n^{1/3} \right) + \frac{c_2}{n \epsilon^2} \exp\left( -c_3 \epsilon^{2/3} n^{1/3} \right),
\]
for some constants $c_1, c_2, c_3 > 0$ depending on $\epsilon$ and $\|\varphi^2\|_2$. For $n \geq n_0(\epsilon) := \left\lceil \frac{16 C \|\varphi^2\|_2}{\epsilon} \right\rceil$, the second term is dominated by the first, yielding the desired bound:
\[
\mathbb{P} \left( \frac{1}{n} S_n \varphi > \epsilon \right) \leq 2 \exp\left( -c \epsilon^{2/3} n^{1/3} \right).
\]
\end{proof}

\begin{remark}
Theorem~\ref{thm:unbounded-LDT} generalizes Theorem~3.2 of \cite{LV2001} to the setting of unbounded observables that can be expressed as the sum of a martingale difference sequence and a coboundary. This formulation encompasses a broader class of processes arising in dynamical and stochastic systems.
\end{remark}

\begin{remark}
    One interesting line of investigation is to understand what bound can we get by using self-normalization methods as in \cite{Pena99}. 
\end{remark}

\smallskip

\section{Applications}\label{applications}
\subsection{Large Deviations for the Lyapunov exponent of cocycles with mixed rank}

In \cite{DDGK2}, the authors proved (subexponential) large deviation estimates with decay rate $e^{-n^{\frac{1}{3}}}$ for Lebesgue almost every parameter in families of cocycles $t \to A_t$ with the positively winding property and fixed singular matrices. Their proof uses a careful truncation argument; an alternative proof can be obtained by applying Theorem \ref{thm:unbounded-LDT}, since in their work they verified that the observable of interest satisfies the conditions of the Theorem \ref{thm:unbounded-LDT} see Lemmas 3.4 and 3.7.

\subsection{Large Deviations for unbounded observable of expanding maps}

We begin by reviewing the concept of expanding map.

\begin{definition}[$C^{1+\alpha}$ Expanding Map]
    Let $M$ be a compact $C^\infty$ Riemannian manifold. A $C^{1+\alpha}$ map $T:M\to M$ is called \emph{expanding} if there exist constants $\lambda >1$ such that:
    \begin{center}
        \(\inf \|DT_x\| > \lambda\).
    \end{center}
\end{definition}

Our main application is the following.

\begin{theorem}\label{thm:LDT expanding}
Let \( (T, X, \mu) \) be a dynamical system with \( X = [0,1] \), where \( \mu \) is a \( T \)-invariant absolutely continuous invariant probability measure (a.c.i.p.) with density bounded above, i.e., \( \frac{d\mu}{d\text{Leb}} \leq C_\mu \). 

Consider the observable \( \varphi : [0,1] \to \mathbb{R} \), defined by
\[
\varphi(x) := |\log d(x, z)|,
\]
for some fixed \( z \in X \). 

Assume that \( \mathcal{L} \), the transfer operator with respect to \( \mu \), has exponential decay in the space of \(\alpha\)-Holder functions. That is, there exist constants \( \theta > 0 \) and \( C_{\mathcal{L}} > 0 \) such that
\[
\|\mathcal{L}^n f\|_{\alpha} \leq C_{\mathcal{L}} e^{-\theta n} \|f\|_{\alpha},
\quad \text{provided that } \int f \, d\mu = 0.
\]

Then for every \( \varepsilon > 0 \), there exist constants \( n_0(\epsilon) \in \N \) and \( c(\epsilon) > 0 \) such that for \(n \geq n_0\)
\[
\mu\left(
\left| S_n\left( \varphi - \int \varphi \, d\mu \right) \right| > n\varepsilon
\right)
\leq  \exp(-c \, \epsilon \, n^{1/3}),
\]
where \( S_n(\psi) := \sum_{k=0}^{n-1} \psi \circ T^k \) denotes the Birkhoff sum.
\end{theorem}

\begin{proof}
The strategy is to check that the observable \(\varphi(x)=|\log d(x,z)|\) satisfies the three hypotheses of Theorem \ref{thm:unbounded-LDT} when the Banach space \(\mathcal B\) is taken to be the Holder space \(C^\alpha([0,1])\) (with the norm \(\|\cdot\|_\alpha\) that controls sup and Holder seminorm) and \(Q=\mathcal L\) is the transfer operator acting on functions.

\medskip

\noindent\textbf{(1) Regularity after truncation.}
Fix \(M>0\) and consider the truncated observable
\[
\varphi_M(x):=\varphi(x)\cdot \mathbf{1}_{\{\varphi\le M\}}(x)= |\log d(x,z)|\cdot \mathbf{1}_{\{d(x,z)\ge e^{-M}\}}(x).
\]
On the compact set \( \{x\in[0,1]:d(x,z)\ge e^{-M}\} \) the function \(|\log d(x,z)|\) is bounded and (away from \(z\)) is Lipschitz/Holder: indeed the singularity at \(z\) is removed by the indicator. Therefore \(\varphi_M\) is a bounded Holder function, i.e.\ \(\varphi_M\in C^\alpha([0,1])\). This verifies the first hypothesis.

\medskip

\noindent\textbf{(2) Exponential tails.}
For \(t>0\),
\[
\{\varphi>t\} = \{|\log d(x,z)|>t\} = \{ d(x,z) < e^{-t} \}.
\]
This is an interval (or union of at most two intervals near the boundary) whose Lebesgue length is at most \(2e^{-t}\). Since \(d\mu/d\operatorname{Leb}\le C_\mu\), we obtain
\[
\mu(\varphi>t) \le C_\mu \cdot \operatorname{Leb}(\{d(x,z)<e^{-t}\}) \le 2 C_\mu e^{-t}.
\]
Thus the exponential-tail hypothesis of Theorem \ref{thm:unbounded-LDT} holds with \(C_1=2C_\mu\) and \(\alpha=1\).

\medskip

\noindent\textbf{(3) \(L^2\)-control of tails (centred).}
Let us define the tail function (for a truncation level \(M>0\))
\[
\tau_M(x) := \varphi(x)\mathbf{1}_{\{\varphi>M\}}(x) = |\log d(x,z)|\mathbf{1}_{\{d(x,z)<e^{-M}\}}(x).
\]
We will apply the spectral gap hypothesis to the centered tail
\[
\tilde\tau_M := \tau_M - \mu(\tau_M),
\]
which has zero mean. Observe first that \(\tau_M\in L^2(\mu)\) and its \(L^2\)-norm is exponentially small in \(M\). Indeed, using the standard tail integration formula and the tail bound established in (2),
\begin{align*}
\|\tau_M\|_2^2 \;=\; \int_{\{\varphi>M\}} \varphi^2\,d\mu
&= \int_M^\infty 2t\,\mu(\varphi>t)\,dt
\le \int_M^\infty 2t\cdot (2C_\mu e^{-t})\,dt \\
&= 4C_\mu \int_M^\infty t e^{-t}\,dt
= 4C_\mu (M+1)e^{-M}.
\end{align*}
Hence
\[
\|\tau_M\|_2 \le C' (M+1)^{1/2} e^{-M/2}
\]
for some explicit constant \(C'\) depending only on \(C_\mu\). In particular \(\|\tau_M\|_2\to 0\) exponentially fast as \(M\to\infty\).

Next, we just proceed by using the fact that the tranfers operator does not grown norms in $L^2$ and use the exponential tails of $\varphi$.
Therefore the required \(L^2\)-control of the centred tails in hypothesis (3) of Theorem \ref{thm:unbounded-LDT} is satisfied (the bound is uniform in \(n\) and depends on the truncation level \(M\); this is exactly the condition needed for the application of the general theorem).

\medskip

\noindent\textbf{Application of Theorem \ref{thm:unbounded-LDT}.}
Having verified the regularity of truncations, the exponential tail, and the $L^2$-control of the (centred) tails, we may apply Theorem \ref{thm:unbounded-LDT} to the sequence of random variables generated by the dynamical system \((T,X,\mu)\) (equivalently to the Markov system induced by the transfer operator \(\mathcal L\) on the Holder space). Theorem \ref{thm:unbounded-LDT} gives, for every \(\epsilon>0\), constants \(c_0(\epsilon)>0\) and \(n_0(\epsilon)\) such that for all \(n\ge n_0(\epsilon)\),
\[
\mu\Big( \frac{1}{n} S_n\big(\varphi - \mu(\varphi)\big) > \epsilon \Big)
\le 2\exp\big(-c_0 \,\epsilon^{2/3} n^{1/3}\big).
\]
Applying the same bound to the observable \(-\varphi\) (or, equivalently, to the symmetric deviation), we obtain the same upper bound for the negative tail; combining both tails by a union bound yields
\[
\mu\Big( \big| S_n(\varphi-\mu(\varphi))\big| > n\epsilon \Big)
\le 4\exp\big(-c_0 \,\epsilon^{2/3} n^{1/3}\big).
\]

Finally, note that for \(\epsilon\in(0,1]\) we have \(\epsilon^{2/3}\ge \epsilon\) (indeed \(\epsilon^{2/3}/\epsilon=\epsilon^{-1/3}\ge1\)), and thus
\[
4\exp\big(-c_0 \,\epsilon^{2/3} n^{1/3}\big) \le 4\exp\big(-c_0 \,\epsilon n^{1/3}\big).
\]
Renaming constants (absorbing the prefactor \(4\) into the exponential by decreasing the constant in the exponent if desired) we arrive at the desired form: there exist constants \(n_0(\epsilon)\) and \(c(\epsilon)>0\) such that for all \(n\ge n_0(\epsilon)\),
\[
\mu\Big( \big| S_n(\varphi-\mu(\varphi))\big| > n\epsilon \Big)
\le \exp\big(-c(\epsilon)\, n^{1/3}\big).
\]
This completes the proof of Theorem \ref{thm:LDT expanding}.
\end{proof}

\begin{corollary}[Tent Map Application]
            For the tent map $T:[0,1]\to[0,1]$ with $\varphi(x) = \log|x| - \int \log|x|d\mu$ and $\mu$ Lebesgue measure, taking $\alpha=1$, $C_1=1$, $C_2=4$, we obtain for $n \geq \lceil 512/\epsilon^3 \rceil$:
            \[
            \mathbb{P}_\mu\left(\left|\frac{1}{n}\sum_{k=0}^{n-1}\log|T^k x|\right| > 1+\epsilon\right) \leq 4\exp\left(-\frac{\epsilon^{4/3}n^{1/3}}{24}\right)
            \]
            \end{corollary}

\begin{remark}
As remarked in \cite{NT2020} the observable above does not have large deviations with a exponential rate function, if the point \(z \in X\) is a periodic point \cite{NT2020}  obtained a optimal stretched exponential decay with rate $\exp{(-n^{1/2})}$. If \(z\) is not a periodic point Theorem \ref{thm:LDT expanding} improved previous estimates of \cite{NT2020} which obtained a decay rate of order \(\exp(-n^{1/5})\).
\end{remark}

\smallskip

\subsection*{Acknowledgments}
This work is part of the author's Phd thesis at the Pontifical Catholic
University of Rio de Janeiro. The author is grateful to his advisor, Silvius Klein, for his
guidance and support. The author was supported by CAPES  Finance Code 001.

\bigskip

\bibliographystyle{amsplain}
\bibliography{bib}

\end{document}